\documentclass[11pt,a5paper]{article}

\usepackage{indentfirst}	
\usepackage[usenames]{color}
\usepackage[latin1]{inputenc}
\usepackage{bm}
\usepackage[english]{babel}
\usepackage{amsfonts, amsmath, amssymb,latexsym,amsthm, graphics,babel}
\usepackage[margin=2cm]{geometry}
\usepackage{geometry,graphicx}
\usepackage{enumerate}
\usepackage[all]{xy}
\usepackage{esint}
\usepackage{breqn}
\usepackage{comment}

\theoremstyle{plain}
\newtheorem*{theorem*}{Theorem}
\newtheorem{theorem}{Theorem}
\newtheorem{lemma}[theorem]{Lemma} 
\newtheorem{corollary}[theorem]{Corollary}
\newtheorem{proposition}[theorem]{Proposition}
\newtheorem{conjecture}[theorem]{Conjecture}

\theoremstyle{definition}
\newtheorem{remark}[theorem]{Remark}

\newtheorem{definition}[theorem]{Definition}

\numberwithin{equation}{section} 
\def\R{\mathbb R}

\def\C{\mathbb C}
\def\H{\mathbb H}
\def\abs#1{\left\vert #1\right\vert}
\newcounter{casenum}

\newcommand{\xs}{X\backslash S}

\begin{document}

\title{Harmonic maps between ideal 2-dimensional simplicial complexes}

\author{Brian Freidin\\
Brown University\\
{\tt bfreidin@math.brown.edu}
\and
Vict\`oria Gras Andreu\\
Brown University\\
{\tt vgras@math.brown.edu}}

\date{}

\maketitle

% ABSTRACT

\begin{abstract}
	We prove existence and regularity results for energy minimizing maps between ideal hyperbolic 2-dimensional simplicial complexes. The spaces in question were introduced by Charitos-Papadopoulos, who describe their Teichm\"uller spaces and some compactifications. This work is a first step in introducing harmonic map theory into the Teichm\"uller theory of these spaces.
\end{abstract}

% INTRODUCTION

\section{Introduction}

In \cite{CP}, Charitos and Papadopoulos study finite 2-dimensional simplicial complexes. They describe how to endow each face (with vertices removed) with the structure of an ideal hyperbolic triangle, and special parameters describing the ways to glue the faces together to form the complex, characterizing those metrics that are complete. They compute the dimension of the Teichm\"uller space of complete ideal hyperbolic metrics on a complex, and describe a compactification of this space in terms of special measured foliations.

The theory of harmonic maps has been applied fruitfully in Teichm\"uller theory in many ways. To a harmonic map $f$, one associates its Hopf differential $\phi(z)dz^2$, the $(2,0)$ part of the pull-back of the target metric by $f$, and the harmonicity of $f$ is equivalent to the holomorphicity of the Hopf differential. Sampson \cite{S} and Wolf \cite{W} show that for a fixed metric $g_0$ on a surface $S$ of genus $g$, the map that associates to a hyperbolic metric $g$ on $S$ the Hopf differential of the harmonic map $f:(S,g_0)\to(S,g)$ is a homeomorphism from the Teichm\"uller space of $S$ to the space of holomorphic quadratic differentials on $(S,g_0)$.

In another direction, Gerstenhaber and Rauch propose in \cite{GR} a variational characterization of Teichm\"uller mappings, those that minimize the complex dilatation in their isotopy class, via harmonic maps. In \cite{K}, in fact, Kuwert shows that the Teichm\"uller map is harmonic with respect to a particular singular flat metric in the conformal class of the target. The aim of this paper is to introduce harmonic map theory into the Teichm\"uller theory of the ideal hyperbolic complexes of \cite{CP}, by first constructing harmonic maps.

Existence results for harmonic maps goes back at least as far as the work of Eells and Sampson \cite{ES} from '64, where they prove existence under curvature and completeness hypotheses by the heat flow method. The study of harmonic maps with singular spaces began with the seminal work of Gromov and Schoen \cite{GS}in '92, who studied harmonic maps into Riemannian simplicial complexes to prove $p$-adic superrigidity. The theory was extended	by Korevaar and Schoen \cite{KS1} to CAT(0) targets, by Eells and Fuglede \cite{EF} to polyhedral domains, and by Jost \cite{J2} to more general metric-measure domains.

For general singular targets, the most regularity one can hope for is Lipschitz continuity (c.f. \cite{GS},\cite{KS1}), and when the domain is simplicial, the map may only be H\"older continuous. In special cases, e.g. harmonic maps from 2-dimensional simplicial complexes into smooth manifolds (c.f. \cite{DM1},\cite{MY}), harmonic maps are smooth or even analytic on the faces of the domain.

While we prove regularity results analogous to those of $(\cite{DM1},\cite{MY})$, they are not a-priori diffeomorphisms. In the classical setting, Schoen-Yau \cite{SY} and independently Jost-Schoen \cite{JS} prove that every diffeomorphism between surfaces is isotopic to a harmonic diffeomorphism, or equivalently that degree one harmonic maps are diffeomorphisms. Jost \cite{J} proves an analogous result for boundary value problems. We follow the techniques of \cite{JS}, but the result that our maps are diffeomorphisms relies on a local statement for neighborhoods in a simplicial complex, which remains at this time a conjecture.

The measured foliations introduced in \cite{CP} to compactify the Teichm\"uller space are characterized up to isotopy by a list of five properties. The Hopf differential of our harmonic maps satisfy none of these, a-priori. The connection between Hopf differentials and the Teichm\"uller theory of ideal hyperbolic complexes also remains at this time unexplored.

The organization of this paper is as follows: In Section~\ref{SBack} we introduce the spaces we are considering and describe the Teichm\"uller space of complete ideal hyperbolic metrics on them. In Section~\ref{SFinite} we construct a finite energy map between two ideal hyperbolic metrics on a given complex $X$.

\begin{theorem*}[c.f. Theorem~\ref{finite energy}]
There exists a map $H\in\mathcal{D}$. That is, $H:(\xs,\sigma)\to(\xs,\tau)$ with finite energy such that $H\vert_T$ is a diffeomorphism for each face $T$ of $X$.
\end{theorem*}

In Section~\ref{SExist} we construct an energy minimizing map. We first describe a harmonic replacement scheme necessary to achieve sufficient continuity for our main minimization argument. We find minimizing maps by a procedure similar to the construction of \cite{KS1} of the solution to the Dirichlet problem. We construct in this section two maps, one minimizing among all maps that respect the simplicial structure of $X$, and the other among all maps respecting the simplicial structure and whose restrictions to each face is a diffeomorphism, following the strategy of \cite{JS} in the classical setting.

\begin{theorem*}[c.f. Theorem~\ref{harmonic}]
    There exist energy minimizing mappings $u\in W^{1,2}$ and $u_D\in\overline{\mathcal{D}}$. That is,
    \[
    E(u) = \inf_{v\in W^{1,2}}E(v),
    \]
    and
    \[
    E(u_D) = \inf_{v\in\mathcal{D}}E(v).
    \]
    Moreover both $u$ and $u_D$ are homotopic to the function $H$ constructed in Theorem~\ref{finite energy}.
\end{theorem*}

In Section~\ref{SProp} we discuss properties of our minimizing maps. In particular, we prove each map is locally Lipschitz continuous (Theorem~\ref{global lip}), proper (Theorem~\ref{proper}), and degree 1 (Theorem~\ref{degree}). We show

\begin{theorem*}[c.f. Theorem~\ref{interior smooth}]
    Let $u\in W^{1,2}$ and $u_D\in\overline{\mathcal{D}}$ be the energy minimizing maps constructed in Theorem~\ref{harmonic}. For any face $T$ of $X$, the restrictions $u\vert_T$ and $u_D\vert_T$ to the interior of $T$ are analytic harmonic maps. Moreover the map $u_D$ is a diffeomorphism on the interior of each face of $X$.
\end{theorem*}

We also prove the map $u\in W^{1,2}$ enjoys more boundary regularity.

\begin{theorem*}[c.f. Theorem~\ref{analytic}]
    For the minimizing map $u\in W^{1,2}$ and any face $T$ of $X$, the restriction $u\vert_T$ to the closed face is analytic up to the boundary.
\end{theorem*}

\subsection*{Acknowledgments}
The authors would like to thank both George Daskalopoulos and Athanase Papadopoulos for suggesting the problem and broader context of the problem, as well as many helpful conversations and suggestions. They also thank Chikako Mese for helpful discussions.

\section{Background}\label{SBack}

\subsection{Simplicial Complexes}

\begin{definition}
A \textit{2-dimensional simplicial complex} is a topological space $X$ together with two finite or infinite sets $\mathcal{C}$ and $\mathcal{F}$ that satisfy the following properties:
\begin{enumerate}
	\item Each $T\in\mathcal{C}$ is a topological triangle. 
		That is, a topological space homeomorphic to a closed 2-dimensional closed disc with three distinguished points on the boundary called vertices. 
		The edges of $T$ are the closed segments of $\partial T$ bounded by two vertices and not containing the third.

	\item $\mathcal{F}$ is a maximal collection of homeomorphisms $f:A\to B$ where $A\subset T$ and $B\subset T'$ are distinct edges of (possibly identical) triangles $T,T'\in\mathcal{C}$ (along with the identity map on each edge of each triangle). For two edges $A,B$, there should be at most one map $f:A\to B$ in $\mathcal{F}$. The collection $\mathcal{F}$ is maximal with respect to two conditions. First, if $f:A\to B$ is in $\mathcal{F}$, then so is $f^{-1}:B\to A$. Second, if $f:A\to B$ and $g:B\to C$ are in $\mathcal{F}$, then so is $g\circ f:A\to C$.

		The elements of $\mathcal{F}$ are called gluing maps.

	\item As a topological space, $X$ is the quotient of the disjoint union $\coprod_{\mathcal{C}}T$ of the triangles in $\mathcal{C}$ by the equivalence relation identifying $x\in A$ with $f(x)\in B$ for each $f:A\to B\in\mathcal{F}$.

		Let $\pi:\coprod_{\mathcal{C}}T\to X$ be the quotient map. The fact that each $f\in\mathcal{F}$ is a homeomorphism and no two maps in $\mathcal{F}$ have the same domain and range implies that $\pi$ is injective on each edge of each triangle in $\mathcal{C}$. The image $\pi(T)$ of a triangle $T\in\mathcal{C}$ is called a \textit{face} of $X$, the image $\pi(e)$ of an edge $e$ of $T$ is called an \textit{edge} of $X$, and the image $\pi(v)$ of a vertex $v$ of $T$ is called a \textit{vertex} of $X$.
		
		We will also always impose an orientation on each edge of $X$. Pulling back by $\pi$, this puts an orientation on each edge $e$ of each triangle $T\in\mathcal{C}$ in such a way that the gluing maps $f\in\mathcal{F}$ are orientation-preserving homeomorphisms.

\item $X$ is path connected.

\item $X$ is locally finite. That is, each edge and each vertex is incident to finitely many faces.
\end{enumerate}
\end{definition}

\begin{remark}
A simplicial complex is said to be \textit{(locally) 1-chainable} if the complement of the vertices, $\xs$, is (locally) connected. By the construction of our complexes, since faces are glued to other faces only along edges (and not at vertices alone), our complexes will satisfy this condition.
\end{remark}

\begin{definition}
A 2-dimensional simplicial complex is \textit{finite} if the number of faces is finite. 
\end{definition}

We will always assume that our complexes are finite.

\begin{definition}
Let $X$ be a 2-dimensional simplicial complex. 
An edge $e$ of $X$ is called \textit{singular} if $e$ is incident to at least three faces. That is, the preimage of $e$ by the quotient map $\pi$ consists of three or more distinct edges.
\end{definition}

\begin{figure}[ht]
\centering
\includegraphics[width=0.5\textwidth]{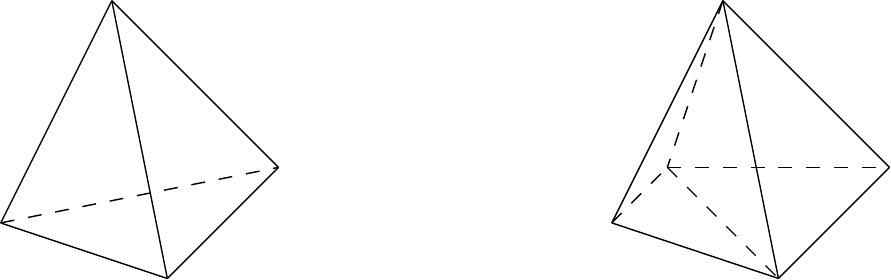}
\caption{2-dimensional simplicial complexes with (right) and without (left) singular edges}
\end{figure}

\subsection{Metrics}\label{Smetrics}

For a 2-dimensional simplicial complex $X$ with a set $S$ of vertices, we want to endow $\xs$ with an ideal hyperbolic metric. This discussion is based on \cite{CP}, though our treatment is slightly altered at some points. There is also a discussion of more general metrics in \cite{CP} of metrics, but we will restrict our attention here.

We will take the upper half-plane model of $\H^2$, that is $\H^2 = \{z\in\C\mid Im(z)>0\} = \{(x,y)\in\R^2\mid y>0\}$ with metric
\[
ds^2 = \frac{dx^2+dy^2}{y^2}.
\]
We will use the real coordinates $(x,y)$ and the complex coordinate $z=x+iy$ almost interchangeably.
We will always view an ideal hyperbolic triangle as the convex hull in $\H^2$ of points $0$, $r\in\R_{>0}$, and $\infty$. Our standard triangle is the convex hull of $0$, $1$, and $\infty$, and will be denoted $\tilde{T}$. Any other triangle is related to $\tilde{T}$ by scaling, i.e. $r\tilde{T}$ is the convex hull of $0$, $r$, and $\infty$.

For each triangle $T\in\mathcal{C}$, choose a diffeomorphism $\phi:T\backslash S\to\tilde{T}$. Use this $\phi$ to give coordinates on $T$, as well as the metric $\phi^*ds^2$. This turns $T\backslash S$ into an isometric copy of the ideal hyperbolic triangle. Henceforth any isometry $\phi:T\backslash S\to r\tilde{T}$ can be used to give coordinates on $T$.

Now each edge $e$ of each triangle $T$ is isometric to the real line $\R$. If there is a gluing map $f:e\to e'$, we require that this map be an (orientation preserving) isometry. The collection of isometries $\{\phi:T\to\tilde{T}\mid T\in\mathcal{C}\}$ together with the isometric gluing maps induces a Riemannian metric in each face of $\xs$, which in turn induces a distance metric on $\xs$ in the usual way, via measuring lengths of paths using the Riemannian metrics in faces.

\begin{definition}
An \textit{ideal hyperbolic structure} or \textit{ideal hyperbolic metric} on $\xs$ is the data of a collection of isometries $\{\phi:T\to\tilde{T}\mid T\in\mathcal{C}\}$ along with isometric gluing maps $f\in\mathcal{F}$, or equivalently the distance metric on $\xs$ that they induce.
\end{definition}

For any triangle $T\in\mathcal{C}$ and any edge $e$ incident to $T$, there is a preferred isometry $\phi:T\to\tilde{T}$ that maps the edge $e$ to the line joining $0$ and $\infty$ in $\H^2$ in an orientation-preserving way (where the line joining $0$ and $\infty$ is oriented towards $\infty$). If we want to view all the faces incident to a common edge of $X$ at the same time, we use the following model.

\begin{definition}[Local model of an edge]
Let $e\in X$ be an edge, and let $\{e_1,\ldots,e_n\}$ enumerate the preimages of $e$ under the quotient map $\pi$. Each edge $e_j$ is an edge of a triangle $T_j\in\mathcal{C}$. For each triangle $T_j$, choose an isometry $\phi_j:T_j\to r_j\tilde{T}$ so that
\begin{enumerate}
	\item $\phi_j$ maps $e_j$ onto the geodesic joining $0$ and $\infty$.
	\item For $f:e_j\to e_k\in\mathcal{F}$, $\phi_j\vert_{e_j} = \phi_k\circ f$.
	\item $\max_jr_j = 1$.
\end{enumerate}
Call the collection of isometries $\{\phi_j:T_f\to r_j\tilde{T}\}$ a \textit{local model for $e$}.
\end{definition}

\begin{remark}
The maps $\phi_j$ constituting a local model for an edge depend on the metric of the complex $X$. Moreover, for a fixed metric, there are two possible local models for a given edge, corresponding to whether $e$ is identified with the geodesic joining $0$ and $\infty$ in an orientation-preserving or orientation-reversing way.
\end{remark}

\subsubsection{Shift Parameters}

For the ideal hyperbolic triangle $\tilde{T}$, there is a distinguished point, called the center of the triangle. It can be obtained as the unique fixed point of the full isometry group $Isom(\tilde{T})\cong S_3$, or as the intersection of the three altitudes of the triangle (see Figure~\ref{distinguished}). For the triangle $\tilde{T}$, the center is at $\frac{1+\sqrt{3}i}{2}$.

There are also distinguished points on each edge of $\tilde{T}$. These points are the feet of the altitudes that define the center. For the triangle $\tilde{T}$ these points are $i$, $1+i$, and $\frac{1+i}{2}$.

\begin{figure}[ht]
\centering
\includegraphics[width=0.25\textwidth]{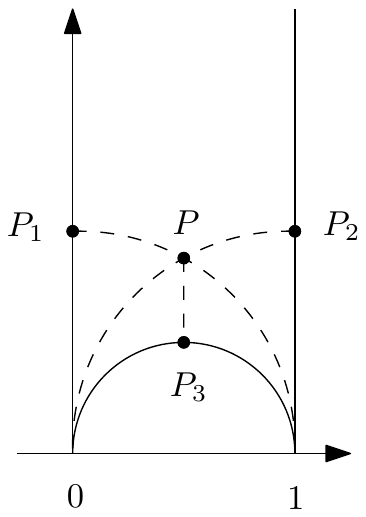}
\caption{Distinguished points.}\label{distinguished}
\end{figure}

Let $\sigma$ be an ideal hyperbolic structure on $\xs$. Let $e$ be an edge of $X$ and let $\{\phi_j:T_j\to r_j\tilde{T}\}$ be a local model for $e$. For two triangles $T_j,T_k$ incident to $e$, the ratio between $r_j$ and $r_k$ measures the distance along $e$ between the distinguished points on $e$ coming from the two triangles, and parametrizes the ways in which the two triangles can be glued.

\begin{definition}
The \textit{shift parameter} associated to a gluing map $f:e_j\to e_k \in \mathcal{F}$ is the real number $\alpha = \alpha_{k,j,e} = \log\left(\frac{r_j}{r_k}\right)$, where $r_j$ and $r_k$ are given by the local model for the edge $e$.
\end{definition}

Before continuing on, let us derive some properties of the shift parameters.

\begin{proposition}\label{shift prop}
Let $X$ be a 2-dimensional simplicial complex and let $\sigma$ be a hyperbolic structure on $\xs$.
Let $T_1, T_2, T_3$ be three faces sharing an edge $e$ in $X$, and let $e_i = \pi^{-1}(e)\cap T_i$ be the edge of $T_i$ corresponding to $e$ for $i=1, 2, 3$. Let $f_{k,j}:e_j\to e_k\in\mathcal{F}$ be the gluing maps for these three edges and let $\alpha_{k,j}$ be the shift parameter of $f_{k,j}$.
Then
\[
\alpha_{1,2}=-\alpha_{2,1}, \qquad \alpha_{1,2}+\alpha_{2,3}+\alpha_{3,1}=0,
\]
\end{proposition}

\begin{proof}
Both of these statements follow from expanding the definition of shift parameters and properties of the logarithm.
\end{proof}

We may now count the dimension of the space of ideal hyperbolic metrics on $\xs$. Fix an edge $e$ of $X$, and let $T_1,\ldots,T_k$ enumerate the faces of $X$ incident to $e$ (with multiplicity). If we specify the shift parameters $\alpha_{j,1,e}$ for $j=2,\ldots,k$, then we can recover $\alpha_{i,j,e}$ for any $1\le i,j\le k$ by Proposition~\ref{shift prop}. Moreover these $k-1$ shift parameters are independent.

So for each edge $e$ we have $deg(e)-1$ independent shift parameters, where $deg(e)$ denotes the number of faces incident to $e$. So in total we have
\[
\sum_e deg(e)-1 = 3F-E
\]
shift parameters. Here $F = \abs{\mathcal{C}}$ denotes the number of faces of $X$, while $E = \abs{\mathcal{E}}$ is the number of edges. While $\sum_e-1=-E$ is clear, $\sum_edeg(e)=3F$ follows since the left hand side is counting pairs $(e,T)$ where $T$ is a face of $X$ and $e$ is an edge incident to $T$. Since each triangle has three edges, this sum is $3F$. But not all of these metrics are complete!

\subsubsection{Completeness}\label{Scompleteness}

If $X$ has no singular edges, then $\xs$ is a surface with punctures. It is well known that a hyperbolic metric on a finite type Riemann surface with punctures is complete if and only if the metric at each puncture is that of a cusp. Thus, it seems natural that a metric $\sigma$ on $X$ will be complete if and only if all the vertices are cusps. In fact this is the content of Proposition 3.2 of \cite{CP}.

A hyperbolic cusp can be constructed as the quotient of $\{z\in\H^2 \mid Im(z)>1\}$ by the group of translations generated by $z\mapsto z+1$. The horizontal curves $t\mapsto z_0+t$ are horocycles centered at $\infty$, and after quotienting by translation they become closed curves around the cusp. Thus a hyperbolic puncture is a cusp if and only if it has a neighborhood foliated by closed horocycles. For more details see \cite{T}, for instance Proposition 3.4.18.

In order to describe this condition for completeness on the complex $X$, we first need to define the link of a vertex. The link of a vertex $v$ of $X$, denoted $\Gamma(v)$, is a graph whose vertices correspond to edges of $X$ incident to $v$, and whose edges correspond to faces of $X$ incident to $v$. Two vertices of $X$ are joined by an edge if the corresponding edges of $X$ cobound the corresponding face. It can also be seen that a neighborhood of $v$ in $X$ is homeomorphic to the cone on $\Gamma(v)$, the set $(\Gamma(v)\times[0,1])/(\Gamma(v)\times\{0\})$. See for instance Figure~\ref{link}.

\begin{figure}[ht]
\centering
\includegraphics[width=0.4\textwidth]{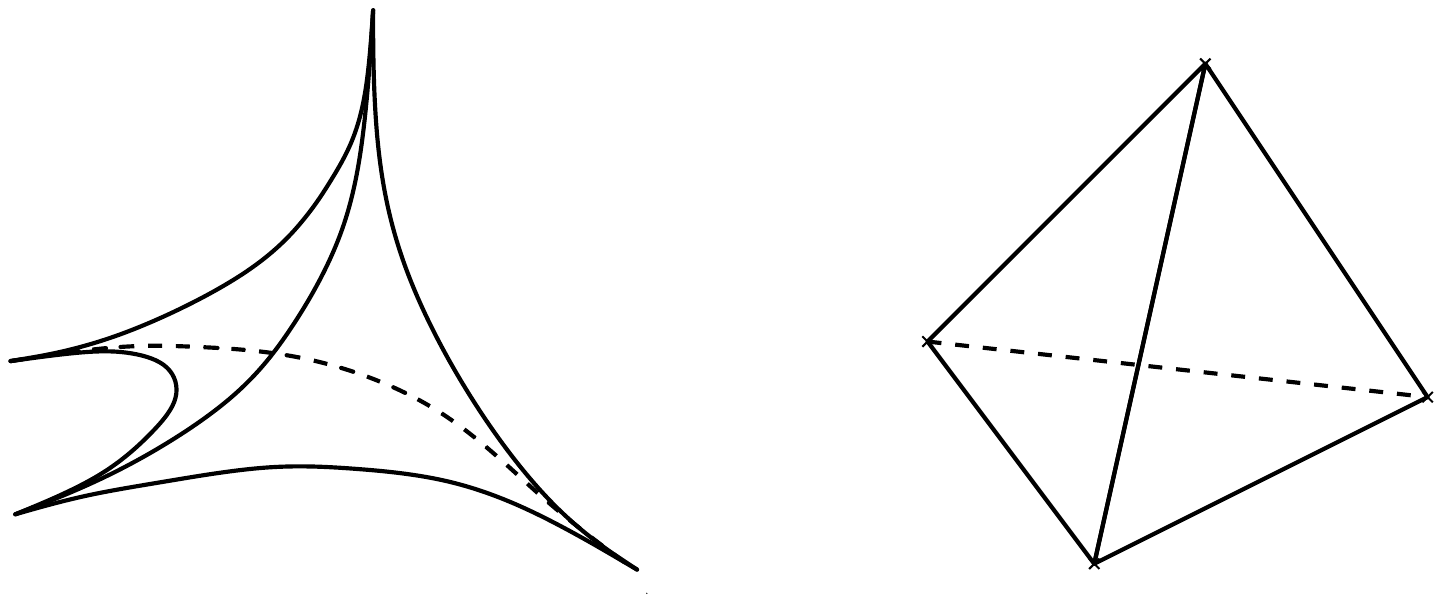}
\caption{2-dimensional simplicial complexes with complete (left) and incomplete (right) metrics}
\end{figure}

\begin{figure}[ht]
\centering
\includegraphics[width=0.5\textwidth]{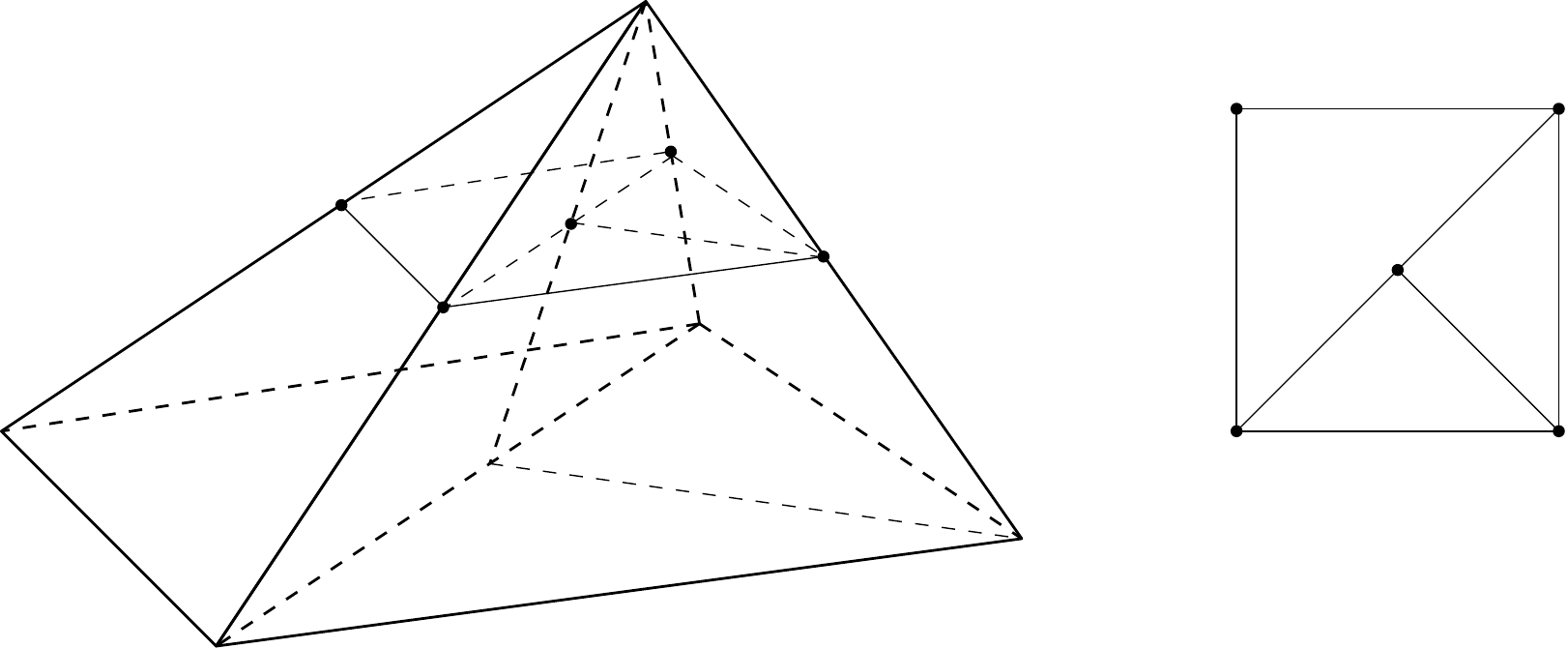}
\caption{Link of a vertex}\label{link}
\end{figure}

To each vertex $z$ in $\Gamma(v)$ we may associate a sign $\epsilon_z = \pm1$, where $\epsilon_z = 1$ if the edge $e$ of $X$ corresponding to $z$ is oriented towards $v$, and $\epsilon_z=-1$ if $e$ is oriented away from $v$. It is possible that an edge of $X$ will have $v$ at each end. In this case there will be two vertices in $\Gamma(v)$ corresponding to the edge, and $\epsilon$ will take value $1$ at one of these vertices and $-1$ at the other.

Let $L$ be a simple closed curve in the graph $\Gamma(v)$ and let us put an orientation on it. The curve $L$ visits a sequence of vertices and edges in $\Gamma(v)$. Say this sequence is $(z_1,f_1,z_2,f_2,\ldots,z_\ell,f_\ell,z_{\ell+1}=z_1)$, where each $z_j$ is a vertex in $\Gamma(v)$ corresponding to an edge $e_j$ in $X$, $f_j$ is an edge in $\Gamma(v)$ corresponding to a face $T_j$ in $X$, and the endpoints of $f_j$ are $z_j$ and $z_{j+1}$. We develop the faces $T_j$ into $\H^2$ as follows.

Let $\phi_1:T_1\to\tilde{T}$ be an isometry so that $\phi_1(e_1)$ is the geodesic joining $0$ and $\infty$ and $\phi_1(e_2)$ is the geodesic joining $1$ and $\infty$. Inductively choose $\phi_j:T_j\to r_j\tilde{T}$ so that $\phi_j\vert_{e_j} = \phi_{j-1}\vert_{e_j}$ and $\phi_j(e_{j+1})$ is to the right of $\phi_j(e_j)$, for $j=2,\ldots,\ell$. The completeness of the metric restricted to the faces $T_1,\ldots,T_\ell$ is equivalent to the horocycles about $\infty$ in the developed picture being closed, which in turn means that $\phi_1(e_1)$ must be glued to $\phi_\ell(e_{\ell+1})$ by a pure real (horizontal) translation. This property is captured by the equation
\[
\sum_{j=1}^\ell \epsilon_j\alpha_{j+1,j} = 0.
\]
Here $\alpha_{j+1,j}$ is the shift parameter associated to the gluing map that glues the triangle $T_j$ to the triangle $T_{j+1}$ along the edge $e_{j+1}$. For more details and a different perspective, see \cite{CP}.

\begin{definition}
Let $X$ be a finite 2-dimensional simplicial complex. The \textit{Teichm\"uller space} of $X$, denoted $\mathcal{T}(X)$, is the space of all complete ideal hyperbolic metrics on $\xs$.
\end{definition}

There is one such relation for each simple closed curve in the link of each vertex. In the link $\Gamma(v)$ of $v$, there are many simple closed curves. The fundamental group $\pi_1(\Gamma(v))$ is free, as $\Gamma(v)$ is a graph, and has rank $\#E(\Gamma(v)) - \#V(\Gamma(v))+1$, where $\#E$ and $\#V$ are the number of edges and vertices, respectively, in $\Gamma(v)$. Thus we have $\#E(\Gamma(v))-\#V(\Gamma(v))+1$ independent relations at each vertex. In total we have
\[
\sum_v\#E(\Gamma(v))-\#V(\Gamma(v))+1 = 3F-2E+V
\]
relations. Imposing these relations on the $3F-E$-dimensional space of hyperbolic metrics, we see that the space $\mathcal{T} = \mathcal{T}(X)$ of \textit{complete} ideal hyperbolic metrics on $X$ has dimension $E-V$.

For a complete metric, we can develop all of the faces incident to a given vertex onto the hyperbolic plane simultaneously and consistently.

\begin{definition}[Local model of a vertex]
Let $v\in X$ be a vertex, and let $\{T_1,\ldots,T_n\}$ enumerate the faces of $X$ incident to $v$ (with repetition if $v$ occurs as the vertex of a face more than once). For each face $T_j$, choose an isometry $\phi_j:T_j\to r_j\tilde{T}$ so that
\begin{enumerate}
	\item $\phi_j$ maps the vertex $v$ to the point at $\infty$.
	\item For $f:e_j\to e_k\in\mathcal{F}$, $\phi_j\vert_{T_j\cap T_k} = \phi_k\vert_{t_j\cap T_k}+t_k$, for some real (horizontal) translation $t_k$.
	\item $\max_jr_j = 1$.
\end{enumerate}
Call the collection of isometries $\{\phi_j\}$ a \textit{local model for $v$}.
\end{definition}

\subsubsection{Curvature}

Here we will derive some curvature properties of the ideal hyperbolic metrics we have defined on $\xs$. First of all, let us recall some definitions.

\begin{definition}
A complete metric space $(Y, d)$ is said to have \textit{curvature bounded from above by} $\kappa$ if the following conditions are satisfied:
\begin{enumerate}[(i)]
	\item $(Y, d)$ is a length space. 
	That is, for any two points $P, Q\in Y$, the distance $d(P,Q)$, which for simplicity we will also write as $d_{PQ}$, is realized as the length of a rectifiable curve $\gamma_{PQ}$ connecting $P$ to $Q$.
	We call such a curve a geodesic.
	\item Let $a=\sqrt{\abs{\kappa}}$. 
	Let $P, Q, R\in Y$ (assume $d_{PQ}+d_{QR}+d_{RP}<\frac{\pi}{\sqrt{\kappa}}$ for $\kappa>0$) with $Q_t$ defined to be the point on the geodesic $\gamma_{QR}$ satisfying $d_{QQ_t}=td_{QR}$ and $d_{Q_tR}=(1-t)d_{QR}$ Then we have
	\[
	\cosh(ad_{PQ_t})\leq \dfrac{\sinh((1-t)ad_{QR})}{\sinh(ad_{PQ})}\cosh(ad_{PQ})+\dfrac{\sinh(tad_{QR})}{\sinh(ad_{QR})}\cosh(ad_{PR})
	\]
	for $\kappa<0$,
	\[
	d^2_{PQ_t}\leq (1-t)d^2_{PQ}+td^2_{PR}-t(1-t)d^2_{QR}.
	\]
	for $\kappa=0$, and
	\[
	\cos(ad_{PQ_t})\geq \dfrac{\sin((1-t)ad_{QR})}{\sin(ad_{QR})}\cos(ad_{PQ})+\dfrac{\sin(tad_{QR})}{\sin(ad_{QR})}\cos(ad_{PR})
	\]
	for $\kappa>0$.
\end{enumerate}
\end{definition}

If $Y$ has curvature bounded from above by 0, we also say that this space is \textit{non-positively curved} (NPC).
If $Y$ has curvature bounded from above by $\kappa\leq 0$, we also say that $Y$ is a CAT($\kappa$) space.

As a consequence of (ii), geodesics in an NPC space are unique, see for instance \cite{KS1}. 
It is also a well-known fact that if $Y$ is locally compact and NPC, then $Y$ is simply connected. 
Conversely, if $Y$ is a complete simply connected Riemannian manifold with non-positive sectional curvature, then $Y$ is an NPC space with the length metric induced from the Riemannian metric.

Furthermore, it is known that any CAT($\kappa$) space is also a CAT($\ell$) space for $\ell\geq\kappa$.
In particular, CAT(-1) spaces are NPC.

\begin{proposition}
	Let $\sigma$ be a complete ideal hyperbolic metric on $\xs$ as defined in Section~\ref{Smetrics}.
	Then $(\xs,\sigma)$ is a local CAT(-1) space.
\end{proposition}

The proof of this result can be found in \cite{CP}, Proposition 1.4.
We can actually improve this result as follows:

\begin{proposition}\label{cat-1}
	Let $\sigma$ be a complete ideal hyperbolic metric on $X$. Let $e$ be an edge of $X$, and let $\{\phi_j:T_j\to r_j\tilde{T}\}$ be a local model for $e$. Then the union
	\[
	\cup_j \phi_j^{-1}\left\{z\in r_j\tilde{T}\mid Re(z)< \frac{r_j}{2}\right\},
	\]
	endowed with the metric of $\sigma$, is a CAT(-1) space.
\end{proposition}

\begin{proof}
	Since each face is isometric to an ideal hyperbolic triangle, it is clear that it is a CAT(-1) space.
	By \cite{Sur}, Chapter 10, Corollary 5 and Lemma 9, a union of a finite number of CAT(-1) spaces glued along a convex set is also CAT(-1). The choice of $Re(z)<\frac{r_j}{2}$ avoids any topology in the set. Since the edges of an ideal hyperbolic triangle are convex, the Proposition follows.
\end{proof}

\section{Existence of finite energy maps}\label{SFinite}

\subsection{Function spaces}

Let $\sigma, \tau$ be complete ideal hyperbolic metrics on $\xs$ as defined in Section~\ref{Smetrics}.

\begin{definition}
A map $u:(\xs,\sigma)\to(\xs,\tau)$ \textit{respects the simplicial structure} of $X$, or is a \textit{simplicial map} if
\begin{itemize}
	\item $u\vert_T$ maps $T$ to $T$ for each face $T$ of $X$, and
	\item $u\vert_e$ maps $e$ to $e$ for each edge $e$ of $X$.
\end{itemize}
\end{definition}

We will consider two classes of maps for the remainder of this paper. The first is
\[
W^{1,2} = \{u:(\xs,\sigma)\to(\xs,\tau)\mid u\text{ is simplicial, and }\forall T\in\mathcal{C}, u\vert_{T}\in W^{1,2}(T,T)\}.
\]
This class consists of all simplicial maps whose restriction to each face of $X$ is a $W^{1,2}$ map. The second class is
\[
\mathcal{D} = \{u\in W^{1,2} \mid \forall T\in\mathcal{C},u\vert_T\text{ is a diffeomorphism}\}.
\]
This class, motivated by the strategy of \cite{JS} for constructing harmonic diffeomorphisms, is the class of simplicial $W^{1,2}$ maps whose restriction to each face is a diffeomorphism.

We will also require the spaces of functions
\[
L^2(\xs,\sigma) = \{f:\xs\to\R \mid \sum_{T\in\mathcal{C}}\int_T f^2 <\infty\},
\]
and
\[
W^{1,2}(\xs,\sigma) = \{f\in L^2(\xs,\sigma) \mid \forall T\in\mathcal{C}, f\vert_T\in W^{1,2}(T)\}.
\]

In all cases, if $A$ is one of the classes defined above, $\overline{A}$ denotes the closure of $A$ with respect to weak convergence.

\subsection{Existence}

The goal of this section is to produce a map in $W^{1,2}$, i.e. a finite energy simplicial map from $(\xs,\sigma)$ to $(\xs,\tau)$. In fact we will produce a map in $\mathcal{D}$.

Recall that if $(X^m,g)$ and $(Y^n,h)$ are Riemannian manifolds, we define the energy of a smooth map $f:X\to Y$ as
\[
E(f):=\int_X\abs{\nabla f}^2d\mu,
\]
where 
\[
\abs{\nabla f}^2(x)=\sum_{i,j=1}^n\sum_{\alpha,\beta=1}^m g^{\alpha\beta}(x)h_{ij}(f(x))\dfrac{\partial f^i}{\partial x^{\alpha}}\dfrac{\partial f^j}{\partial x^{\beta}}
\]
with $x^{\alpha}$ and $f^i$ the local coordinate systems around $x$ and $f(x)$ respectively.
$\abs{\nabla f}^2$ is called the energy density.

Each face of $X$, endowed with either metric $\sigma$ or $\tau$, can be identified with the ideal hyperbolic triangle $\tilde{T}$ in an isometric way. Using this identification, a map $f:X\to X$ restricted to a face $T\in\mathcal{C}$ can be viewed as a map $f\vert_T:\tilde{T}\to\tilde{T}$. We can use the coordinates on $\tilde{T}$ as well as its hyperbolic metric to compute the energy $E(f\vert_T)$ of the map $f$ on the face $T$. The total energy of $f:X\to X$ is the sum over all the faces of $X$, namely

\[
E(f) = \sum_{T\in\mathcal{C}}E(f\vert_T).
\]

\begin{theorem}\label{finite energy}
There exists a map $H\in\mathcal{D}$. That is, $H:(\xs,\sigma)\to(\xs,\tau)$ with finite energy such that $H\vert_T$ is a diffeomorphism for each face $T$ of $X$.
\end{theorem}

\begin{proof}
The biggest difficulty in constructing such a finite energy map is controlling the energy of the map around the punctured vertices of $X$.
Indeed, let $v$ be a vertex of $X$, $\Gamma(v)$ its link and $N(v)$ a neighborhood of $v$ homeomorphic to a cone on $\Gamma(v)$ and foliated by horocycles.
If the energy of a map in $N(v)$ is finite for all $v\in S$, then we can extend the map to all of $X$ more easily since $X\backslash\left(\cup_{v\in S}N(v)\right)$ is compact.

Also note that, since the complexes are finite and the volume of each face is equal to $\pi$, the volume of $N(v)$ is finite for each $v\in S$ (in fact the volume of $\xs$ is finite). 
Thus, any map with bounded energy density will have finite energy.

Fix $v\in S$. Let $\{\phi_{j,\sigma}:T_j\to r_{j,\sigma}\tilde{T}\}$ be a local model for $v$ with respect to the metric $\sigma$, and similarly $\{\phi_{j,\tau}:T_j\to r_{j,\tau}\tilde{T}\}$ for the metric $\tau$. Consider the neighborhoods $N_*(v) = \{\phi_{j,*}^{-1}(z)\in r_{j,*}\tilde{T} \mid Im(z)>2\}$ for metrics $*=\sigma,\tau$. Map $N_\sigma(v)$ to $N_\tau(v)$ as follows: for $x\in N_\sigma(v)\cap T_j$, define
\[H_v(x) = \phi_{j,\tau}^{-1}(\phi_{j,\sigma}(x)).\]

By definition of the local models of vertices, the map $H_v$ is well defined on $N_\sigma(v)$. Note also that this map is a smooth diffeomorphism on the interior of each $N_\sigma(v)\cap T_j$ and injective on the intersection of each edge with $N_\sigma(v)$. In the coordinates defined by the local models, the map $H_v\vert_{T_j}$ is simply the identity, whose energy density is 2. Hence the energy of $H_v$ on $N_\sigma(v)$ is finite, as there are finitely many faces incident to $v$.

Now define $H_w$ in $N_\sigma(w)$ for each vertex $w$ in $X$. The neighborhoods $N_\sigma(w)$ were chosen so that neighborhoods of distinct vertices do not intersect. Since there are finitely many vertices, we have used only finitely much energy in our construction thus far. We will initially define our map $H$ to coincide with $H_v$ on $N_\sigma(v)$ for each vertex $v$, but may later modify it near the horocyclic boundaries of these neighborhoods.

Let us now extend this map to the compact remainder of $X$. Fix an edge $e$ with endpoints $v$ and $w$. On $e\cap N_\sigma(v)$ let $H = H_v$ and on $e\cap N_\sigma(w)$ let $H = H_w$. On the remaining segment of $e$ define $H$ so that $H\vert_e$ is a smooth diffeomorphism. Do the same construction on each edge of $X$.

Now let $X'\subset X$ be a compact subset of $X$ so that for each triangle $T_j$, $X'\cap T_j$ is convex with respect to the metric $\tau$, with smooth boundary consisting of segments of edges of $T_j$ and curves within the neighborhoods $N_\tau(v)$ around the vertices of $T_j$. See for example Figure~\ref{convex}. To construct $X'$, let's work in the ideal triangle $\tilde{T}$. First note that the edges of $\tilde{T}$ are geodesic and hence convex. We connect the edges meeting at the point $0$ smoothly by a convex curve. We may move this curve to the other punctures by a  M\"obius transformation. By choosing these connecting curves to be close enough to the punctures, when we put them into $T_j$ by $\phi_j^{-1}$ the curves will lie in the neighborhoods $N_\tau(v)$ and their endpoints will match with the other faces sharing an edge.

\begin{figure}[ht]
	\centering
	\includegraphics[width=0.3\textwidth]{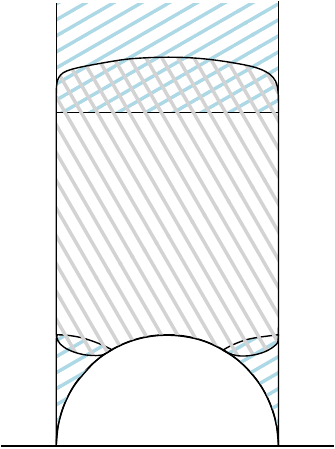}
	\caption{$\mathcal{C}^\infty$ convex set containing a hexagon}\label{convex}
\end{figure}

Let $\gamma$ be a curve in $X\cap T_j$ so that $H(\gamma) = \partial(X'\cap T_j)$. Since $H$ is only defined along edges and the neighborhoods $N_\sigma(v)$ so far, the curve $\gamma$ must lie in segments along the edges of $T_j$ as well as curves in the neighborhoods $N_\sigma(v)$. On the edges and on these neighborhoods $H$ is a smooth diffeomorphism, so $\gamma$ is a smooth curve bounding a compact region $\Omega_j\subset T_j$ and $H\vert_\gamma$ is a smooth diffeomorphism of $\gamma$ onto $\partial(X'\cap T_j)$.

Construct the solution $H_j:\Omega_j\to X'\cap T_j$. That is, $H_j$ is a harmonic map with $H_j\vert_\gamma = H\vert_\gamma$. By Proposition 1 of \cite{JS}, the map $H_j$ is a diffeomorphism on the interior of $\Omega_j$. Now on the interior of the face $T_j$ we define $H$ to coincide with $H_j$ inside $\Omega_j$ and to coincide with $H_v$ in $N_\sigma(v)\backslash\Omega_j$ for each vertex $v$ of $T_j$. Note that this redefines $H$ on a small part of the neighborhoods $N_\sigma(v)$. We can make this definition on each face of $X$, and since we have already fixed the values of $H$ on the edges of $X$, the map $H$ is continuous.

On the face $T_j$, the map $H$ coincides with $H_v$ on a neighborhood of the vertex $v$, so its energy near the punctures is finite. On $\Omega_j$ the map is harmonic and certainly has finite energy there. Since there are finitely many faces in $X$, the map $H$ has finite energy in total.

In the face $T_j$, $H$ is a diffeomorphism on the interior of $\Omega_j$ and on the interior of each $N_\sigma(v)\backslash\Omega_j$, but it may fail to be differentiable over $\partial\Omega_j$. But using Theorem 1.2 of \cite{IKO}, we know that we can approximate this map by diffeomorphisms with energy bounded by the energy of $H$ in $W^{1,2}$.  Making this approximation in each face, we have a finite energy map with all of the desired properties.
\end{proof}

\section{Existence of minimizing maps}\label{SExist}

In this section, we will show the existence of two maps from $(\xs, \sigma)$ to $(\xs, \tau)$, one of which minimizes energy among the class $W^{1,2}$ of finite energy simplicial maps, and the other minimizes energy on the class $\mathcal{D}$ of maps that are a diffeomorphism on each face of $X$.

In order to achieve sufficient regularity in our minimization scheme, we will have to do a harmonic replacement argument. We will first describe our harmonic replacement scheme. We will use Proposition 1 of \cite{JS}, so we will need some convex sets with smooth boundary in the target.

We construct a sequence of sets $X_n$ in the same way as the set $X'$ from the proof of Theorem~\ref{finite energy}, so that for each $n$, $X_n\cap T_j$ is convex with smooth boundary in $T_j$ for each face $T_j$ of $X$. We also require $X_n\subset X_{n+1}$ and $\cup_{n=1}^\infty X_n = \xs$. These sets $X_n$ can be constructed for either the domain or the target metric.

Before we proceed we'll need some preliminary results.

\begin{lemma}\label{Wcomplete}
	Let $G\in W^{1,2}$ be given.
	Let $\{u_k\}_k\subset W^{1,2}$  be an $L^2$ Cauchy sequence with $u_k\equiv G$ on $X\backslash X_n$ for each $k$. Suppose the energies of $u_k$ are uniformly bounded. Then the sequence converges in $L^2$ to a function $u\in W^{1,2}$ with $u\equiv G$ on $X\backslash X_n$.
\end{lemma}

\begin{proof}
	Let $\{u_k\}_k\subset W^{1,2}$ be an $L^2$ Cauchy sequence with uniformly bounded energies.
	Recall the definition (c.f. \cite{KS1} of the $L^2$ distance between maps from $(\xs,\sigma)$ to $(\xs,\tau)$:
	\[
	d^2(u,v) = \sum_{T\in\mathcal{C}}\int_T d_\tau^2(u(x),v(x))d\mu_\sigma(x).
	\]
	We know that the space of $L^2$ maps from $\tilde{T}$ to itself is a complete metric space. Since the sequence $u_{k,j} = u\vert_{T_j}$ is $L^2$-Cauchy, it converges to a function $u_{0,j}:T_j\to T_j$.
	Define 
	\[
	u(x)=u_{0,j}(x), \quad x\in T_j.
	\]
	In order to define $u$ at the boundary of $T_{j, n}$ we'll use Theorem 1.12.2 of \cite{KS1}.
	It states that each $u_{k,j}\vert_{\partial T_j}$ is a well-defined map in $L^2(\partial T_j, \xs)$.
	Moreover, since the energy of $u_{k,j}$ is bounded uniformly in $k$, the sequence $u_{k,j}\vert_{\partial T_j}$ converges in $L^2(\partial T_j, \xs)$ to $u_{0,j}\vert_{\partial T_j}$.
	Since for faces $T_j,T_\ell$ meeting at an edge $e$ we have $u_{k,j}\vert_e = u_{k,\ell}\vert_e$, the same remains true in the limit. That is, $u_{0,j}\vert_e = u_{0,\ell}\vert_e$, so the map $u$ is continuous over the edges of $X$.
	
	Since each $u_k$ satisfies $u_k\equiv G$ on $X\backslash X_n$, the same is true for $u$. And since the energy is a lower semi-continuous function (c.f. \cite{KS1}, Theorem 1.6.1), the $L^2$ limit $u$ of the sequence $u_k$ is in $W^{1,2}$.
\end{proof}

\begin{theorem}\label{poincare}
	\textbf{(Poincar\'e inequality for complexes).} 
	Let $K$ be a bounded connected compact subset of $X$ and $f\in W^{1,2}(\xs,\sigma)$ with compact support in $K$. 
	Then, there exists a constant $C$ only depending on $K$ and $X$ such that
	\[
	\int_K u^2d\mu_{\sigma}\leq C\int_K \abs{\nabla u}^2d\mu_{\sigma}.
	\]
\end{theorem}

The proof of this Theorem can be found in \cite{DM2}, Theorem 2.6.

\subsection{Harmonic Replacement}

Here we will construct harmonic maps on compact subsets of $\xs$ to use in our later arguments.

\begin{theorem}\label{dirichletW}
	Given $G\in W^{1,2}$ and a compact set $X'$ where $X'\cap T$ is convex with smooth boundary in $T$ for each face $T$, there exists $u\in W^{1,2}$ minimizing energy among all maps $v\in W^{1,2}$ with $v\equiv G$ on $X\backslash G^{-1}(X')$.
\end{theorem}

\begin{proof}
	Consider the class of maps $W^{1,2}_G$ of maps $u\in W^{1,2}$ with $u\equiv G$ on $X\backslash G^{-1}(X')$. We suppress the dependence of $W^{1,2}_G$ on the set $X'$ in the notation.
	Let $\{v_k\}\subset W^{1,2}_G$ be an energy minimizing sequence. 
	
	First replace $v_k$ with $u_k$, where $u_k = v_k = G$ on $X\backslash G^{-1}(X')$, and for each face $T$, $u_k$ solves the Dirichlet problem on $T\cap G^{-1}(X')$ with boundary data given by $v_k$. As $u_k$ solves a Dirichlet problem where it doesn't coincide with $v_k$, the energy of $u_k$ is no more than the energy of $v_k$, and $u_k$ also lies in $W^{1,2}_G$ so it is also a minimizing sequence.
	
	We will now show that $\{u_k\}$ is an $L^2$-Cauchy sequence. Define $u_{k,\ell}(x)$ to be the midpoint of the geodesic joining $u_k(x)$ and $u_\ell(x)$.
	Note that $u_{k,\ell}$ is well defined. 
	If $x$ lies in the interior of a face $T$, then since $u_k(x), u_\ell(x)\in T$ and $T$ is NPC, there exists a unique geodesic joining these two points and therefore $u_{k,\ell}$ is well defined. 
	If $x$ lies in an edge $e$, $u_k(x)$ and $u_\ell(x)$ also lie in that edge and therefore $u_{k,\ell}(x)\in e$ since edges are geodesics.
	Finally, it is clear that $u_{k,\ell}=G$ on $G^{-1}(X')$ since both $u_k$ and $u_\ell$ do.
	It is also easy to check that $u_{k,\ell}\in L^2$.
	
	From \cite{KS1} we have equation (4.1):
	\begin{equation}\label{bound}
	2E(u_{k,\ell})\leq E(u_k)+E(u_\ell)-\dfrac{1}{2}\int_{G^{-1}(X_n)}\abs{\nabla d_{\sigma}(u_k,u_\ell)}^2\leq 2K.
	\end{equation}	
	Since $\{u_k\}$ is a minimizing sequence, from \eqref{bound} we see that 	
	\[
	\lim_{i,j\to\infty}\int_{X_n}\abs{\nabla d_{\sigma}(u_i,u_j)}^2=0.
	\]
	Since $d_{\sigma}(u_k, u_\ell)\in W^{1,2}(\xs,\sigma)$ with compact support (Theorem 1.12.2 \cite{KS1}), we can use the Poincar\'e inequality of Theorem~\ref{poincare} and therefore the sequence $\{u_k\}$ is an $L^2$-Cauchy sequence. Moreover the energies of $u_k$ are uniformly bounded by $E(u_k)\leq E(G)$.
	Thus, by Lemma \ref{Wcomplete}, $\{u_k\}$ converges to a function $u\in W^{1,2}_G$. 
	
	By semicontinuity of the energy, 
	\[
	E(u)\leq \liminf_k E(u_k)= \inf_{v\in W^{1,2}_G}E(v).
	\]
	Since $u\in W^{1,2}_G$, we have $E(u) = \inf_{v\in W^{1,2}_G}E(v)$ as desired.

\end{proof}

A similar prove yields the same result for maps in the class $\overline{\mathcal{D}}$.

\begin{theorem}\label{dirichletD}
	Given $G\in\mathcal{D}$ and a compact set $X'$ as in Theorem~\ref{dirichletW}, there exists $u\in \overline{\mathcal{D}}$ minimizing energy among all maps in $v\in\overline{\mathcal{D}}$ with $v\equiv G$ on $G^{-1}(X')$.
	Moreover, for each face $T$, the restriction of $u$ to the interior of $T\cap G^{-1}(X)$ is a harmonic diffeomorphism.
\end{theorem}

\begin{proof}
	Consider the class $\mathcal{D}_G$ of maps $u\in\mathcal{D}$ with $u\equiv G$ on $G^{-1}(X')$ (again suppressing the dependence on $X'$ in the notation). Let $\{v_k\}\subset\overline{\mathcal{D}}_G$ be an energy minimizing sequence. We may in fact assume that $v_k\in\mathcal{D}_G$ by an approximation argument, as in Theorem 1.2 of \cite{IKO}.
	
	As in the proof of Theorem~\ref{dirichletW}, replace $v_k$ by $u_k$, where $u_k = v_k = G$ on $G^{-1}(X')$ and for each face $T$, $u_k$ solves the Dirichlet problem on $T\cap G^{-1}(X')$ with boundary data given by $v_k$. Since $T\cap X'$ is convex in the target metric with smooth boundary and the boundary values given by $G$ are a homeomorphism, Proposition 1 of \cite{JS} implies that $u_k$ restricted to the interior of $T\cap G^{-1}(X')$ is a harmonic diffeomorphism. 
	
	The remainder of the proof of Theorem~\ref{dirichletW} proves that the sequence $u_k$ constructed in this proof converges in $L^2$ to a function $u\in W^{1,2}_G$ with $E(u) = \inf_{v\in\mathcal{D}_G}E(v)$.

	Moreover, since energy bounded sequences of $W^{1,2}$ are sequentially weakly compact, $u$ is the weak $W^{1,2}$ limit of functions in $\mathcal{D}_G$, we have $u\in\overline{\mathcal{D}}_G$.
	
	For each face $T$, the map $u$ restricted to the interior of $T\cap G^{-1}(X')$ is a limit of harmonic diffeomorphisms, and hence harmonic. Since it is harmonic with respect to its own boundary values, and those boundary values are monotonic (as a limit of diffeomorphisms), the restriction of $u$ to the interior of $T\cap G^{-1}(X')$ is a harmonic diffeomorphism, by Theorem 5.1.1 of \cite{J}.
\end{proof}

\begin{proposition}\label{homotopic}
	The energy minimizing maps in $W^{1,2}_G$ and $\overline{\mathcal{D}}_G$ constructed in the previous theorems are homotopic to $G$. 	
\end{proposition}

\begin{proof}
	Recall that any two continuous maps whose target is an NPC space are homotopic by geodesic homotopy. That is, for maps $u_0$ and $u_1$, connect $u_0(x)$ and $u_1(x)$ by the unique geodesic, and let $u_t(x)$ be the point a fraction $t$ of the way along this geodesic. The homotopy $u_t$ connects the maps $u_0$ and $u_1$.
		
	Since all maps involved are simplicial, and each open face and each edge of $X$ is NPC, the restrictions of $u$ and $G$ to each face and each edge are homotopic.
	Furthermore these homotopies are all compatible.
\end{proof}

\subsection{Local Regularity}

Here we will establish enough regularity for the harmonic replacements constructed above so that we may use it in our global construction.

For a point $p_0\in (\xs,\sigma)$, let $\text{st}(p_0)$ denote the union of all closed faces containing $p_0$. We will always assume that a ball centered at $p_0$ is contained in $\text{st}(p_0)$. There are two types of balls in $\xs$ that we will consider, and we may take Euclidean balls since the hyperbolic metric in the triangle $\tilde{T}$ is conformally Euclidean.
\begin{enumerate}[(i)]
	\item If $p_0$ lies in the interior of a face $T$, let $\phi:T\to\tilde{T}$ be an isometry. Take the Euclidean ball $B_r(\phi(p_0))$, a topological disk.
	\item If $p_0$ lies on an edge $e$, let $\{\phi_j:T_j\to r_j \tilde{T}\}$ be a local model for $e$. If there are $N$ faces incident to $e$ (with multiplicity) then take the $N$ Euclidean half discs $B_r(\phi_j(p_0))\cap r_j\tilde{T}$ identified along their common boundary diameter.
\end{enumerate}

\begin{theorem}\label{3.9}
Let $p_0\in e$ be a point in an edge of $(\xs,\sigma)$, and let $u:B_r(p_0)\to(\xs,\tau)$ be minimizing among maps that map $e\cap B_r(p_0)$ to $e$ and $T_j\cap B_r(p_0)$ to $T_j$ for each face $T_j$, and have the same trace on $\partial B_r(p_0)$. For local models $\{\phi_{j,\alpha}:T_j\to r_{j,\alpha}\tilde{T}\}$ for $e$ with respect to the metrics $\alpha = \sigma,\tau$, let $f_j = f_j^1+if_j^2 = \phi_{j,\tau}\circ u\circ\phi_{j,\sigma}^{-1}$ represent $u\vert_{T_j\cap B_r(p_0)}$ in coordinates. Define the Hopf differential
\[
\varphi_j = \left(\abs{\frac{\partial f_j}{\partial x}}^2 - \abs{\frac{\partial f_j}{\partial y}}^2 - 2i\left\langle\frac{\partial f_j}{\partial x},\frac{\partial f_j}{\partial y}\right\rangle\right)dz^2.
\]
Then $\varphi_j$ is holomorphic in the interior of $T_j\cap B_r(p_0)$, and moreover for all $y$ where defined,
\[
Im\sum_{j=1}^N\varphi_j(iy) = 0.
\]
\end{theorem}

The vanishing of the Hopf differential is a well-known fact for harmonic maps, and the proof of the balancing formula can be found in \cite{DM1}, Theorem 3.9.

\begin{theorem}\label{3.10}
	If $u:B_r(p_0)\to (\xs, \tau)$ is minimizing in the same sense as in the previous Theorem, then with respect to the Euclidean metric in $B_r(p_0)$, for each $j=1,\ldots,N$,
	\[
	\abs{\frac{\partial f_j}{\partial y}}^2(\bar{x}+i\bar{y})\leq \dfrac{2}{\pi r^2}E(u) \quad \text{and}\quad \abs{\frac{\partial f_j}{\partial x}}^2(\bar{x}+i\bar{y})\leq \dfrac{2N+2}{\pi r^2}E(u),
	\]
	where $r$ is no more than the distance from $(\bar{x},\bar{y})$ to $\partial B_{1,N}$.
\end{theorem}

The proof of this theorem can also be found in \cite{DM1}, Theorem 3.10.

\begin{remark}
    The proof of Theorem~\ref{3.9} relies on a domain variation argument involving precomposing the map with a diffeomorphism of the domain. Hence the result also applies if we require the map $u$ minimizes among diffeomorphisms (i.e. in the class $\mathcal{D}$). The proof of Theorem~\ref{3.10} relies only on the harmonicity of the map on the interior of each half disc and the results of Theorem~\ref{3.9}, so this result too applies if we require the map $u$ minimizes among diffeomorphisms.
\end{remark}

\begin{theorem}\label{local lip}
	Let $G\in W^{1,2}$, $X'\subset(\xs,\tau)$ compact as in Theorem~\ref{dirichletW}, and let $u\in W^{1,2}$ be the energy minimizing map constructed in Theorem~\ref{dirichletW}.
	Let $V$ be a compact subset in the interior of $G^{-1}(X')$.
	For a face $T$ and a point $p\in V\cap T$, let $\phi:T\to\tilde{T}$ be an isometry that sends $p$ to the left of the line $\left\{\frac{1}{2}+is\mid s>0\right\}$ and above the horocycle $\left\{s+i\frac{\sqrt{3}}{2}\mid0\le s\le1\right\}$.
	Let $q = q(p) \in \{is \mid s\ge1\}$ be the point closest to $\phi(p)$ with respect to the Euclidean distance in $\tilde{T}$.
	Let $\rho = \frac{\sqrt{5}-1}{2}$ and assume the Euclidean ball of center $q$ and radius $\rho$, intersected with $\tilde{T}$ lies in the interior of $\phi(G^{-1}(X'))$ for every $p\in V$.
	Then there is a constant $C_T$ depending only on $V\cap T$ so that
	\[
	\abs{\nabla u}^2\leq C_TE(u)
	\]
	with respect to the hyperbolic metric in $V\cap T$. As a result, $u$ is Lipschitz continuous in $V$, with Lipschitz constant depending only on $V$ and the energy of $u$.
\end{theorem}

\begin{proof}
	First note that for any $p\in T$, the map $\phi$ always exists, since the three possible horocyclic neighborhoods of the punctures that we might consider cover $\tilde{T}$ (see Figure~\ref{horneigh}).
	Let $d_e$ denote the Euclidean distance in $\tilde{T}$.
	
\begin{figure}[ht]
\centering
\includegraphics[width=0.25\textwidth]{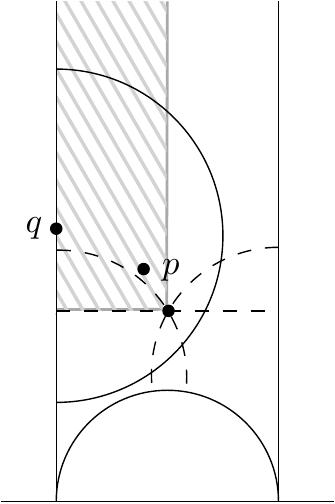}
\caption{Location of $\phi(p)$ and $q$ in $\phi(V\cap T)$}\label{horneigh}
\end{figure}
	
	Consider a ball $B$ around the point $\phi^{-1}(q)$ so that $\phi(B\cap T) = B_\rho(q)\cap\tilde{T}$, where $B_\rho(q)$ is the Euclidean ball. The Euclidean distance from $\phi(p)$ to $\partial B_\rho(q)$ is at least some positive distance $\delta$. Since $u$ is minimizing in $B$ we may apply Theorem~\ref{3.10} to obtain
	\[
	\abs{\nabla u}^2(p) = \abs{\dfrac{\partial u}{\partial x}}^2(p)+\abs{\dfrac{\partial u}{\partial y}}^2(p)	\leq C\int_B\abs{\nabla u}^2d\mu \le E(u).
	\]
	
	But this estimate is with respect to the Euclidean metric in $B_\rho(q)$. If $\phi(p) = x+iy$ in coordinates, then with respect to the hyperbolic metric we get
	\[
	\abs{\nabla u}^2(p) \le Cy^2E(u).
	\]
	Since $V$ is compact, there is a maximum value of $y$ that occurs for points in $V\cap T$. So there is a constant $C_T$ so that $\abs{\nabla u}^2\le C_TE(u)$ over all of $V\cap T$.
	
	Taking a maximum over all the faces of $X$, we see that $\abs{\nabla u}^2$ can be bounded on all of $V$ by some constant times the energy of $u$. Since the energy density dominates the partial derivatives, this gives a Lipschitz bound depending only on $V$ and the energy of $u$.
\end{proof}

\begin{remark}
	The same proof also hold for the energy minimizing map constructed in Theorem~\ref{dirichletD}.
\end{remark}

\subsection{Global Existence}\label{Sexistence}

Now we may prove the global existence theorem.

\begin{theorem}\label{harmonic}
    There exist energy minimizing mappings $u\in W^{1,2}$ and $u_D\in\overline{\mathcal{D}}$. That is,
    \[
    E(u) = \inf_{v\in W^{1,2}}E(v),
    \]
    and
    \[
    E(u_D) = \inf_{v\in\mathcal{D}}E(v).
    \]
    Moreover both $u$ and $u_D$ are homotopic to the function $H$ constructed in Theorem~\ref{finite energy}.
\end{theorem}

\begin{proof}
    The proofs for $W^{1,2}$ and for $\mathcal{D}$ are essentially the same, so we will do them simultaneously. In the case of $\mathcal{D}$, if we choose a minimizing sequence $\{u_k\}$ in $\overline{\mathcal{D}}$  we can approximate by $v_k\in\mathcal{D}$, and since energy is lower semi-continuous this implies $\inf_{v\in\mathcal{D}}E(v) = \inf_{v\in\overline{\mathcal{D}}}E(v)$. So when we take a minimizing sequence in a moment, we may take the sequence in $\mathcal{D}$.
    
    Let $A$ be either the space $W^{1,2}$ or $\mathcal{D}$, and let $\{u_k\}\subset A$ be an energy minimizing sequence. That is,
    \[
    E(u_k)\searrow \inf_{v\in A}E(v).
    \]
    Each $u_k$ has energy bounded by $E(u_1)$, and by the argument in Proposition~\ref{homotopic} each $u_k$ is homotopic to $H$.
    
    Let $X_n$ be a compact exhaustion of $(\xs,\sigma)$ and $X'_n$ a compact exhaustion of $(\xs,\tau)$ as described in the beginning of this section. In particular, for each face $T$ of $X$, $X'_n\cap T$ should be convex with smooth boundary in $T$ with respect to the metric $\tau$. We will do a harmonic replacement on the sequence $u_k$ using Theorems~\ref{dirichletW} and \ref{dirichletD}.
    
    Fix a map $u_k$ in the minimizing sequence to play the role of $G$ in the statement of either Theorem~\ref{dirichletW} or Theorem~\ref{dirichletD}. Choose $m = m(k)$ large enough so that $K_k\subset u_k^{-1}(X'_m)$ can play the role of $V$ in Theorem~\ref{local lip}. For $G = u_k$ and $X' = X'_m$, let $v_k$ be the map produced by Theorem~\ref{dirichletW} if $A = W^{1,2}$, or by Theorem~\ref{dirichletD} if $A = \mathcal{D}$. Since $v_k$ solves a Dirichlet problem where it doesn't coincide with $u_k$, $E(v_k)\le E(u_k)$ and so $v_k$ is also a minimizing sequence, and still homotopic to $H$ by Proposition~\ref{homotopic}.
    
    In order to show existence of an energy minimizing map in $W^{1,2}$ we'll use the Arzel\`a-Ascoli Theorem and a diagonal argument. Arzel\`a-Ascoli states that a sequence of maps from a compact metric space to a complete metric space is relatively compact by the uniform norm if and only if the sequence is equicontinuous and pointwise bounded. We first establish equicontinuity.
    
    Since each $v_k$ is energy minimizing in $X_k$ (with respect to its boundary values and the class of maps $A$), it is certainly energy minimizing in $X_1$. The set $X_1\subset X_k\subset u_k^{-1}(X'_m(k))$ is small enough to play the role of $V$ in Theorem~\ref{local lip}, so we apply this Theorem to show that $v_k$ is Lipschitz continuous in $X_1$, with Lipschitz constant depending only on $X_1$ and the energy of $v_k$, which is bounded by $E(u_1)$. Hence the sequence $v_k$ is uniformly (in k) Lipschitz inside of $X_1$, and hence equicontinuous in $X_1$.
    
    To see that $\{v_k\}$ is pointwise bounded in $X_1$ we must show for every $x_0\in X_1$ that there exists $y_0\in \xs$ and $R>0$ independent of $k$ so that $d_\tau(v_k(x_0),y_0)\leq R$. Suppose for contradiction that for some $x_0\in X_1$ the sequence $\{v_k(x_0)\}$ is not pointwise bounded, so there is a subsequence $\{v_{k_\ell}(x_0)\}$ such that $v_{k_\ell}(x_0)$ escapes to a punctured vertex of $X$. The point $x_0$ lies in some face $T$, which we identify with the ideal hyperbolic triangle via an isometry $\phi:T\to\tilde{T}$. Without loss of generality we may assume that $\phi(v_{k_\ell}(x_0))$ escapes to the point at $\infty$.
	
	Let $e$ be the edge opposite $\infty$, the hyperbolic geodesic joining $0$ and $1$. Fix $x_1\in \phi^{-1}(e)\cap X_1$. Since the sequence $v_{k_\ell}$ is uniformly Lipschitz continuous in $X_1$ we have
	\begin{eqnarray*}
	d(\phi(v_{k_\ell}(x_0)),e) & = & \inf_{y\in e}d(\phi(v_{k_\ell}(x_0)),y)\\
	 & \le & d_\tau(\phi(v_{k_\ell}(x_0)),\phi(v_{k_\ell}(x_1)))\\
	 &\le & C d_\sigma(x_0,x_1).
	\end{eqnarray*}
	Here we have used $d_\alpha$ to denote the hyperbolic distance function $T$ with respect to a metric $\alpha = \sigma,\tau$. Since $d(\phi(v_{k_\ell}(x_0)),e)$ is uniformly bounded in $\ell$, it cannot have escaped to $\infty$ as we supposed, a contradiction. Hence the sequence $\{v_k\}_k$ is pointwise bounded in $X_1$ as desired.
	
	Now we know $\{v_k\}\subset A$ is equicontinuous and pointwise bounded in $X_1$.
	By the Arzel\`a-Ascoli Theorem, there exist a subsequence $\{v_{1,k}\}$ that converges uniformly on compact sets of $X_1$.
	Repeating the argument and construction above we find that there exists a further subsequence $\{v_{2,k}\}$ that converges uniformly on compact subsets of $X_2$. And in fact, for each $\ell$, we find a subsequence $\{v_{\ell,k}\}_k$ of $\{v_{\ell-1,k}\}_k$ that converges uniformly on compact subsets of $X_\ell$.
	We take the diagonal subsequence $\{v_{k,k}\}$, and this converges uniformly on all compact subsets of $\xs$ to a continuous function $u\in\overline{A}$.
	Since the convergence is uniform in compact subsets of $\xs$, $u$ is also homotopic to $H$.
	
	Finally, note that by lower semi-continuity of the energy, 
	\[
	E(u)\leq\lim_{k\to\infty}E(u_k) = \inf_{v\in A}E(v).
	\]
	Thus $E(u) = \inf_{v\in A}E(v)$ as desired.
\end{proof}

Although, the analytic techniques of both \cite{SY} and \cite{JS} appear too sensitive to apply near the noncompact punctures and singular edge of $\xs$, we nonetheless believe that the analogous statement about harmonic diffeomorphisms holds in this setting.

\begin{conjecture}
The maps $u\in W^{1,2}$ and $u_D\in\overline{\mathcal{D}}$ coincide. In other words, the restriction $u\vert_T$ of $u$ to any face $T$ of $X$ is a diffeomorphism.
\end{conjecture}

\section{Properties of minimizing maps}\label{SProp}

In this section we will collect various properties of the minimizing map we constructed in Section~\ref{Sexistence}. First, we see that our map is locally Lipschitz. The same proof as for Theorem~\ref{local lip} proves the following:

\begin{theorem}\label{global lip}
	Let $u\in W^{1,2}$ and $u_D\in\overline{\mathcal{D}}$ be the energy minimizing maps constructed in Theorem~\ref{harmonic}.
	Then in any compact subset $V$ of $\xs$ there is a constant $C$ depending only on $V$ so that
	\[
	\abs{\nabla u}^2(p)\leq CE(u)\quad\text{ and }\quad\abs{\nabla u_D}^2(p)\leq CE(u_D)
	\]
	in $V$. As a result $u$ and $u_D$ are locally Lipschitz continuous, with Lipschitz constant at a point $p\in \xs$ depending only on the energy of $u$ (resp. $u_D$) and the distance of $p$ from the center of any face in which $p$ lies.
\end{theorem}

\begin{remark}\label{euclid lip}
    For a face $T$ of $X$ and an isometry $\phi:T\to\tilde{T}$, let $f = \phi\circ u\circ\phi^{-1}$ represent $u\vert_T$ in coordinates. Then in the set $\{z\in\tilde{T}\mid Im(z)\geq 1\}$, the proof of Theorem~\ref{local lip} witnesses that the Euclidean norm of $\nabla f$, denoted $\abs{\nabla f}^2_e$ satisfies
    \[
    \abs{\nabla f}^2_e \le CE(u).
    \]
    The same is true replacing the map $u$ by the map $u_D$.
\end{remark}

In fact, these maps are much more regular in the interior of each face.

\begin{theorem}\label{interior smooth}
    Let $u\in W^{1,2}$ and $u_D\in\overline{\mathcal{D}}$ be the energy minimizing maps constructed in Theorem~\ref{harmonic}. For any face $T$ of $X$, the restrictions $u\vert_T$ and $u_D\vert_T$ to the interior of $T$ are analytic harmonic maps. Moreover the map $u_D$ is a diffeomorphism on the interior of each face of $X$.
\end{theorem}

\begin{proof}
    Fix a face $T$ and an isometry $\phi:T\to\tilde{T}$, and let $f,f_D$ represent $u,u_D$ in these coordinates. For any disc $D$ in the interior of $\tilde{T}$, the maps $f$ and $f_D$ are both in Lipschitz (c.f. Theorem~\ref{global lip}), so they have well defined boundary values on $\partial D$. Since both $f$ and $f_D$ minimize energy in $D$ with respect to their own boundary values, they must be harmonic in $D$. It is well known that harmonic maps between analytic manifolds are analytic in the interior, see for instance \cite{M}, Theorem 6.8.1.
    
    To show the map $f_D$ is a diffeomorphism in the interior of $\tilde{T}$ we argue as in \cite{JS}, Theorem 1. Consider the minimizing sequence $v_k$ constructed in the proof of Theorem~\ref{harmonic}, after the harmonic replacement has been performed. let $f_k:\tilde{D}\to\tilde{T}.$ represent the map $v_k\vert_T$. Fix a disc $D$ in the interior of $\tilde{T}$, and fix $k_0$ so that $D\subset X_{k_0}$. Since each $f_k$ is a harmonic diffeomorphism in $D$ for $k>k_0$, Since the convergence of $f_k$ to $f_D$ is uniform on $D$, the limit is also a harmonic diffeomorphism inside $D$. Now $f_D$ is a local diffeomorphism, and combining this with the result below (Theorem~\ref{degree}) that $f_D$ has degree 1, $f_D$ is a diffeomorphism on $\tilde{T}$.
\end{proof}

We now collect some topological properties satisfied by both $u$ and $u_D$, namely that these maps are proper and have degree 1. First, though, we state the following version of the Courant-Lebesgue Lemma, the proof of which can  be found in \cite{J}, Lemma 3.1.1.

\begin{lemma}\label{courantlebesgue}
    Let $D$ be the unit disc in the plane, $x_0\in D$, $N$ a Riemannian manifold with distance function $d$, and let $\delta<1$.
	Let $u\in W^{1,2}(D, N)$.
	Suppose $\partial B(x_0, r)\cap D$ is connected for all $r\in (\delta, \sqrt{\delta})$.
	Then there exists some $r\in(\delta, \sqrt{\delta})$ for which $u\vert_{\partial B(x_0, r)\cap D}$ is absolutely continuous and 
	\[
	d(u(x_1), u(x_2))\leq\left(\dfrac{8\pi K}{\log 1/\delta}\right)^{1/2},
	\]
	for all $x_1, x_2\in\partial B(x_0, r)\cap D$.
\end{lemma}

    Before we can show that the maps $u,u_D$ are proper on the whole complex, we first show they are proper on each edge of the complex.

\begin{proposition}\label{properedge}
    Let $u\in W^{1,2}$ and $u_D\in\overline{\mathcal{D}}$ be the energy minimizing maps constructed in Theorem~\ref{harmonic}. For each edge $e$ of $X$, both $u\vert_{e\backslash S}$ and $u_D\vert_{e\backslash S}$ are proper and surjective.
\end{proposition}

\begin{proof}
	We will show that if a sequence of points on an edge escapes to a puncture, then the image of that sequence escapes to the same puncture. We will state the proof for the map $u$, but the same argument holds, word for word, for the map $u_D$.
	
	Fix an edge $e$ of $X$ and let $p_n\in e$ be a sequence escaping to a punctured vertex $v$.
	Choose a face $T$ of $X$ incident to $e$ and identify $T$ with the ideal hyperbolic triangle $\tilde{T}$ so that $v$ is identified with the point at $\infty$.
	Let $f:\tilde{T}\to\tilde{T}$ represent $u\vert_T$ in these chosen coordinates.
	Extend $f$ to a neighborhood of the vertical boundaries of $\tilde{T}$, e.g. by using values of $u$ on adjacent faces of $X$, or by setting $f$ to be constant on horizontal rays emanating out of $\tilde{T}$.
	
	Now the sequence $p_n$ is identified with a sequence $z_n=iy_n\in\mathbb{C}$ with $y_n\to\infty$.
	We first claim that there is an unbounded sequence $\tilde{z}_n = i\tilde{y}_n$ whose image under $f$ is unbounded and satisfies $d_e(z_n,\tilde{z}_n)<\frac{1}{2}$, where $d_e$ denotes the Euclidean distance function in $\mathbb{C}$.
	
	Indeed, consider the Euclidean ball $B_{\frac{1}{2}}(\frac{1}{2}+iy_n)$, which is tangent to the vertical boundaries of $\tilde{T}$. By Lemma~\ref{courantlebesgue} there is some radius $r_n\in\left(\frac{1}{2},\sqrt{\frac{1}{2}}\right)$ so that
	\begin{equation}\label{CLe}
	d(u(t_1),u(t_2)) \le C\sqrt{E(f,D_n)}
	\end{equation}
	for some constant $C$ and for all $t_1,t_2\in\partial B_{r_n}(\frac{1}{2}+iy_n)$. Here $E(f,D_n)$ denotes the energy of $f$ on $D_n$, the Euclidean unit ball centered at $\frac{1}{2}+iy_n$. Since the hyperbolic metric is conformally Euclidean, the energy $E(f,D_n)$ does not depend on which metric is used to compute it.

	Since $r_n>\frac{1}{2}$, there are points $z_n^j\in \partial B_{r_n}(\frac{1}{2}+iy_n)\cap\{x=j\}$ for $j=0,1$.
	Since the energy of $u$ is finite, $E(f,D_n)$ tends to $0$ as $n\to\infty$.
	By \eqref{CLe}, $d(f(z_n^0),f(z_n^1))$ also tends to $0$ as $n\to\infty$.
	Since $u$ maps each edge to itself, the only way for sequences on different edges to approach each other is if they both approach the point at $\infty$.	Hence $f(z_n^j)\to\infty$.
		
	Also, by the Pythagorean Theorem (see Figure \ref{compmet}):
	$$
	d_e^2(z_n, z_n^0) = d_e^2\left(z_n^0, \frac{1}{2}+iy_n\right) - d_e^2\left(z_n, \frac{1}{2}+iy_n\right) =r_n^2-\left(\dfrac{1}{2}\right)^2\leq \dfrac{1}{2}-\dfrac{1}{4}=\dfrac{1}{4}.
	$$

	\begin{figure}[ht]
		\centering
		\includegraphics[width=0.4\textwidth]{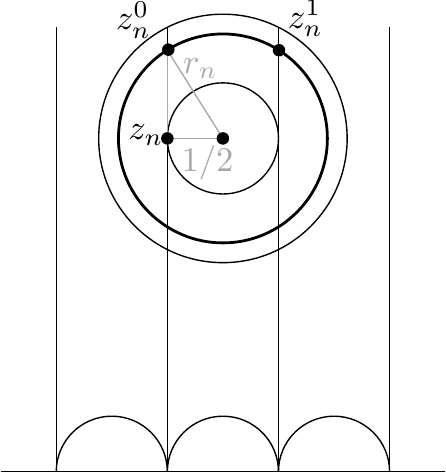}
		\caption{Courant-Lebesgue Lemma}\label{compmet}
	\end{figure}
	
	Thus $\tilde{z}_n\to\infty$, $f(\tilde{z}_n)\to\infty$, and $d_e(z_n,\tilde{z}_n)<\frac{1}{2}$.
	We complete the proof of properness by showing that $d(f(z_n),f(\tilde{z}_n))$ is uniformly bounded, implying that $f(z_n)\to\infty$ as well.
	Remark~\ref{euclid lip} establishes that $\abs{\nabla f}_e^2\le CE(u)$ in a neighborhood of the vertex at $\infty$. Here $\abs{\nabla f}_e^2$ is the norm of $\nabla f$ taken with respect to the Euclidean metric on $\C$.
	
	Now we have the estimate
	\[
	d(u(z_n),u(\tilde{z}_n)) \le CE(u)d_e(z_n,\tilde{z}_n)\le\frac{1}{2}CE(u).
	\]
	Hence $d(u(z_n),u(\tilde{z}_n))$ is uniformly bounded, as desired.
	
	Surjectivity of $u$ now follows easily by a continuity argument. 
	Since $u\vert_e$ is a proper continuous map, its image must be an interval.
	And since $u$ maps each end of $e$ to itself, it's clear that $u\vert_e$ must be surjective.
\end{proof}

% PROPER

\begin{theorem}\label{proper}
The energy minimizing maps $u\in W^{1,2}$ and $u_D\in\overline{\mathcal{D}}$ constructed in Theorem~\ref{harmonic} are proper.
\end{theorem}

\begin{proof}
Again we state the proof for the map $u$, but the same argument holds for $u_D$. Let $p_n$ be an unbounded sequence in $X\backslash S$. Since there are finitely many faces and vertices of $X$, we may assume (after passing to a subsequence) that our sequence lies in a single face $T$ and accumulates on a single vertex $v$ of $T$. Now choose an isometry $\phi:T\to\tilde{T}$ sending $v$ to the point at $\infty$ and write $f = \phi\circ u\circ\phi^{-1}$ to represent $u\vert_T$ in these coordinates. If we let $z_n = x_n+iy_n = \phi(p_n)$ then we have $y_n\to\infty$.

Consider the sequence $\tilde{z}_n = iy_n$. From Proposition~\ref{properedge} we know $f(\tilde{z}_n)\to\infty$. Since $d_e(z_n,\tilde{z}_n)\le1$ (here $d_e$ denotes the Euclidean distance in $\mathbb{C}$) and $\abs{\nabla f}^2_e\le CE(u)$ (where $\abs{\nabla f}^2_e$ denotes the norm of $\nabla f$ taken with respect to the Euclidean metric), we conclude
\[
d(u(z_n),u(\tilde{z}_n))\le CE(u)d(z_n,\tilde{z}_n) \le CE(u).
\]
And hence $u(z_n)\to\infty$ as well.
\end{proof}

% DEGREE 1

\begin{theorem}\label{degree}
The energy minimizing maps $u\in W^{1,2}$ and $u_D\in\overline{\mathcal{D}}$ have degree 1.
\end{theorem}

\begin{proof}
Let $F:(X\backslash S, \sigma)\times[0,1]\to(X\backslash S, \tau)$ be the homotopy such that $F(x, t)$ is the point which is the fraction $t$ of the way along the geodesic from $H(x)$ to $u(x)$, where $H$ is the finite energy map constructed in Theorem~\ref{finite energy}. Since both $H$ and $u$ are proper maps that have the same behavior near each puncture, $F$ is a proper homotopy.

Now Theorem 13.2.1 and property (ii) in section 13.2 of \cite{DFN} imply that, since $H$ and $u$ are properly homotopic and the degree of $H$ is 1 (since $H$ is a diffeomorphism), the degree of $u$ is also 1. The same also holds for the map $u_D$.
\end{proof}

\subsection{Regularity of the $W^{1,2}$ minimizer}

Here we collect properties of the minimizing map $u\in W^{1,2}$. We begin with a balancing condition that characterizes the minimizer.

\begin{definition}\label{balancing}
For an edge $e$ of $X$, let $\{\phi_{j,\sigma}:T_j\to r_{j,\sigma}\tilde{T}\}$ denote a local model for $e$ with respect to the metric $\sigma$, and similarly $\{\phi_{j,\tau}:T_j\to r_{j,\tau}\tilde{T}\}$ for the metric $\tau$. Let $u:(\xs,\sigma)\to(\xs,\tau)$ be a map, and let $f_j = \phi_{j,\tau}\circ u\circ\phi_{j,\sigma}^{-1}:r_{j,\sigma}\tilde{T}\to r_{j,\tau}\tilde{T}$ represent $u\vert_{T_j}$ in the induced coordinates. Write $f_j = f_j^1 + if_j^2$.

We say $u$ satisfies the \textit{weak balancing condition} along $e$ if for any test function $\varphi\in C^\infty_c(0,\infty)$ we have
\[
\sum_j\int_0^\infty\varphi(y)\frac{\partial f_j^2}{\partial x}(iy)dy = 0.
\]
We say that $u$ satisfies the \textit{strong balancing condition} along $e$ if
\[
\sum_j\frac{\partial f_j^2}{\partial x}(iy) = 0
\]
pointwise.
\end{definition}

\begin{proposition}
The minimizing map $u\in W^{1,2}$ constructed in Theorem~\ref{harmonic} satisfies the weak balancing condition.
\end{proposition}

\begin{proof}
After changing notation, this is the content of Theorem 3 of \cite{DM3}. The result follows from the first variation formula for the energy functional.
\end{proof}

\begin{remark}
    In fact, a converse to the above proposition is true. Investigating the first variation formula for the energy functional shows that a map $u:(\xs,\sigma)\to(\xs,\tau)$ is energy minimizing if and only if $u\vert_T$ is harmonic for each face $T$ and $u$ satisfies the weak balancing condition on each edge $e$.
\end{remark}

As a result of the balancing condition, the minimizing map $u\in W^{1,2}$ enjoys a great deal of regularity.

\begin{proposition}\label{c1beta}
Fix an edge $e$ of $X$, and let $\{\phi_{j,\sigma}:T_j\to r_{j,\sigma}\tilde{T}\}$ and $\{\phi_{j,\tau}:T_j\to r_{j,\tau}\tilde{T}\}$ be local models for $e$ with respect to the metrics $\sigma$ and $\tau$.
Let $u\in W^{1,2}$ be the minimizing map from Theorem~\ref{harmonic}, and let $f_j = f_j^1+f_j^2 = \phi_{j,\tau}\circ u\circ\phi_{j,\sigma}^{-1} :r_{j,\sigma}\tilde{T}\to r_{j,\tau}\tilde{T}$ represent the map $u$ in coordinates. Let $y>0$ and $0<r<\min_jr_j$. There exists a neighborhood $\Omega$ of $iy$ with compact closure in $B_r(iy)$ so that $f_j^\alpha\in C^{1,\beta}(\Omega\cap r_{j,\sigma}\tilde{T})$ for each $j$ and for each $\alpha=1,2$.
\end{proposition}

\begin{proof}
We follow the strategy of Theorem 4 from \cite{DM3}. Suppose there are $J$ faces meeting at the edge $e$, so the indices $j$ range from $1$ up to $J$. Fix an index $j_0$. For the function $f_{j_0}^2$, define
\[
\psi^2(x+iy) = \begin{cases}
f_{j_0}^2(x+iy), & x\ge0\\
-f_{j_0}^2(-x+iy) + \frac{2}{J}\sum_{j=1}^J f_j^2(-x+iy), & x<0.
\end{cases}
\]
For the function $f_{j_0}^1$, define
\[
\psi^1(x+iy) = \begin{cases}
f_{j_0}^1(x+iy), & x\ge0\\
-f_{j_0}^1(-x+iy), & x<0.
\end{cases}
\]

In either case the Lipschitz continuity of $u$ implies that $\psi^i$ is Lipschitz. Next we claim that for any test function $\varphi\in C^\infty_c(B_r(iy))$,
\[
\lim_{\epsilon\to0}\int_{B_r(iy)\cap \{x=\epsilon\}}\varphi\frac{\partial\psi^\alpha}{\partial x} = \lim_{\epsilon\to0}\int_{B_r(iy)\cap \{x=-\epsilon\}}\varphi\frac{\partial\psi^\alpha}{\partial x}.
\]
For $\alpha = 1$, use the fact that $f_{j_0}^1(it)=0$ for all $t$. For $\alpha = 2$, use the balancing condition.

From here the proof follows almost verbatim from \cite{DM3}. The functions $\psi^\alpha$ satisfy an elliptic equation with controlled coefficients, and the standard elliptic theory implies the regularity claimed.
\end{proof}

\begin{corollary}
The minimizing map $u\in W^{1,2}$ from Theorem~\ref{harmonic} satisfies the strong balancing condition.
\end{corollary}

\begin{proof}
This follows immediately from the weak balancing condition and the regularity of Proposition~\ref{c1beta}.
\end{proof}

\begin{corollary}
Let $u:(X\backslash S,\sigma)\to(X\backslash S,\tau)$ be a harmonic map. Then the restriction of $u$ to the closure of any face is $C^\infty$ smooth.
\end{corollary}

\begin{proof}
In the interior of a face we know $u$ is analytic from Theorem~\ref{interior smooth}. All that remains to check is that $u$ is smooth up to the edges. This result follows from the bootstrapping argument of \cite{DM3}, Corollary 6, using the functions $\psi^\alpha$ defined in the proof of our Proposition~\ref{c1beta}.
\end{proof}

Moreover, the map $u\in W^{1,2}$ is analytic all the way up to the boundary! We begin by reviewing some important concepts in the theory of partial differential equations.

Let $G$ be a bounded domain with sufficiently smooth non empty boundary and let
\begin{align}\label{linearpde}
L_{jk}(x, D)u_k(x)= f_j(x), \qquad j=1,..., N, \quad x\in G,
\end{align}
be a linear system of partial differential equations subject to boundary conditions of the form
\begin{align}\label{linearbdry}
B_{rk}(x, D)u_k(x)=g_r(x), \qquad r=1,..., m, \quad x\in\partial G,
\end{align}
for some $m$ that will be specified below.
Assume integer weights are given to the system: $s_1,..., s_N$ associated to the equations in the interior, $t_1,..., t_N$ associated to the functions $u_1,..., u_N$, and $h_1,..., h_m$ associated to the boundary equations.
These weights must be chosen so that the order of $L_{jk}$ is less that $s_j+t_k$ and the order of $B_{rk}$ is less than $t_k-h_r$.
Moreover, we can assume $\max s_j= 0$.

Let $L^0_{jk}(x, \Xi), B^0_{rk}(x, \Xi)$ be the principal part of the operator $L_{jk}$ and $B_{rk}$ respectively, that is, the terms of order exactly $s_j+t_k$ and $t_k-h_r$ respectively, and $L(x, \Xi)$ be the determinant of the matrix whose components are $L_{jk}^0(x, \Xi)$.

\begin{definition}
	The system \eqref{linearpde} is \textit{elliptic} if and only if $L(x, \Xi)\neq 0$ for any real non-zero $\Xi$. 
\end{definition}

Note that $L$ is a homogeneous polynomial of degree 
$$
P=\sum_{j=1}^N(s_j+t_j).
$$

\begin{definition}
	The system \eqref{linearpde} is \textit{properly elliptic} if and only if $P$ is even, say $P=2m$ and, for each pair $\Xi, \Xi'$ of linearly independent vectors the equation
	$$
	L(x, \Xi+\tau\Xi')=0
	$$
	has $m$ roots with positive imaginary part and $m$ roots with negative imaginary part.
\end{definition}

Let $x_0\in\partial G$, $n$ be the unit normal at $x_0$ and $\zeta$ any real vector tangent to $\partial G$ at $x_0$. 
Let $\tau_s^+(x_0, \zeta)$, $s=1,..., m$, be the roots of $L(x_0, \zeta+\tau n)=0$ with positive imaginary part (which exist if the system is properly elliptic).
Set
$$
L_0^+(x_0, \zeta; \tau)=\prod_{s=1}^m(\tau-\tau_s^+(x_0, \zeta)),
$$
And let $L^{jk}(x_0, \zeta+\tau n)$ be the components of the matrix adjoint to the matrix $L_{jk}^0(x_0, \zeta+\tau n)$.
Define
$$
Q_{rk}(x_0, \zeta; \tau)=\sum_{j=1}^NB_{rj}^0(x_0, \zeta+\tau n)L^{jk}(x_0, \zeta+\tau n).
$$
\begin{definition}
	For any $x_0\in \partial G$ and any real vector $\zeta$ tangent to $\partial G$ at $x_0$, let us regard $L^+_0(x_0, \zeta; \tau)$ and the elements of the matrix $Q_{rk}(x_0, \zeta; \tau)$ as polynomials in $\tau$.
	The system \eqref{linearbdry} of boundary operators satisfies the \textit{complementing condition} (with respect to the system \eqref{linearpde}) if and only if the rows of $Q$ are linearly independent modulo $L_0^+(x_0, \zeta; \tau)$.
	That is
	$$
	\sum_{r=1}^mC_rQ_{rk}(x_0, \zeta; \tau)\equiv 0 \pmod{L^+_0}
	$$
	only if the $C_r$ are all 0.
\end{definition}

We extend the above definitions to a nonlinear system of differential equations of the form
\begin{align}\label{nonlinear}
\varphi_k(x, Du)=0, \quad k=1,..., N,\quad x\in G,
\end{align}
with boundary values
\[
\chi_r(x, Du)=0, \quad r=1,..., m, \quad x\in \partial G,
\]
in which the $\varphi_k$ and $\chi_r$ are analytic functions of their arguments. 

\begin{definition}\label{nonlin}
The system \eqref{nonlinear} is \textit{elliptic} along the solution $u$ if the linear equations of variation 
$$
L_{jk}(x, D)v^k:=\left.\dfrac{d}{d\lambda}\right|_{\lambda = 0}\varphi_j(x, Du+\lambda Dv)=0,
$$
form a linear elliptic system. 
We define \textit{properly elliptic} and the \textit{complementing condition} in the same way. 
\end{definition}

To prove that the minimizing $u\in W^{1,2}$ is analytic up to an edge $e$ of $X$, fix local models $\{\phi_{j,*}:T_j\to r_{j,*}\tilde{T}\}$ for the edge $e$ with respect to the metrics $*=\sigma,\tau$, and represent $u\vert_{T_j}$ in coordinates by $f_j = f_j^1+if_j^2 = \phi_{j,\tau}\circ u\circ\phi_{j,\sigma}^{-1}:r_{j,\sigma}\tilde{T}\to r_{j,\tau}\tilde{T}$. Suppose that there are $N$ faces, $T_1,\ldots,T_N$, incident to $e$, counted with multiplicity.

\begin{theorem}\label{analytic}
    For the minimizing map $u\in W^{1,2}$ and any face $T$ of $X$, the restriction $u\vert_T$ to the closed face is analytic up to the boundary.
\end{theorem}

\begin{proof}
We will use the representations $f_j$ described above in a local model for an edge $e$ to show that $u$ is analytic up to the edge $e$. In particular, we will show that each $f_j$ is analytic up to the edge $\{x=0\}$ of $\tilde{T}$.

We will use the boundary regularity of \cite{M}, Theorem 6.8.2. The equations governing the maps $f_j$ in the interior are
\begin{align}
    \frac{\partial^2 f_j^1}{\partial x^2} + \frac{\partial^2 f_j^1}{\partial y^2} - \frac{2}{f_j^2}\left(\frac{\partial f_j^1}{\partial x}\frac{\partial f_j^2}{\partial x} + \frac{\partial f_j^1}{\partial y}\frac{\partial f_j^2}{\partial y}\right) & = 0\\
    \frac{\partial^2 f_j^2}{\partial x^2} + \frac{\partial^2 f_j^2}{\partial y^2} + \frac{1}{f_j^2}\left[\left(\frac{\partial f_j^1}{\partial x}\right)^2 + \left(\frac{\partial f_j^1}{\partial y}\right)^2 - \left(\frac{\partial f_j^2}{\partial x}\right)^2 - \left(\frac{\partial f_j^2}{\partial y}\right)^2\right] & = 0.\nonumber
\end{align}
The leading order behavior of each equation is already linear, so the principle symbols fit into a matrix $L^0(x,\Xi) = \abs{\Xi}^2Id_{2N}$. The boundary conditions come in several types. First, the fact that each $f_j$ maps the edge $\{x=0\}$ to itself implies that
\[
f_j^1(0,y) = 0 \qquad \text{for }j=1,\ldots,N.
\]
And the fact that all the $f^j$ must agree on that edge implies
\[
f_j^2(0,y) = f_1^2(0,y) \qquad \text{for }j=2,\ldots,N.
\]
And finally the balancing condition says
\[
\sum_{j=1}^N\frac{\partial f_j^2}{\partial x}(0,y) = 0.
\]
Note that this system of boundary conditions is linear. Label these equations 1 through $2N$ in order, and name functions $\tilde{f}_j = f_j^1$, $\tilde{f}_{j+N} = f_J^2$ for $j=1,\ldots,N$. Then the nonzero entries in the matrix $B$ encoding the boundary conditions are
\begin{align*}
B_{rr} = 1, &\qquad r = 1,\ldots,N\\
B_{r,N+1} = 1, &\qquad r=N+1,\ldots,2N-1\\
B_{r,r+1} = -1, &\qquad r=N+1,\ldots,2N-1\\
B_{2N,k} = \frac{\partial}{\partial x}, &\qquad k=N+1,\ldots,2N.
\end{align*}
 
We assign weights $t_k = 2$ to each function and $s_j = 0$ to each equation, so that the order of $L_{j,k}$ is $2 = s_j+t_k$. We assign weight $h_r=2$ for boundary conditions $r=1,\ldots,2N-1$ and $h_{2N} = 1$, so that the order of each $B_{r,k}$ is equal to $t_k-h_k$.

To show the system is elliptic, it suffices to show that for $\Xi\neq 0$, the matrix $L^0(x,\Xi)$ is invertible. Since $\abs{\Xi}^2\neq 0$ and $Id_{2N}$ is invertible, the system is elliptic.

To show the system is properly elliptic, we must show the polynomial $det(L^0(x,\Xi+\tau\Xi'))$ has $N$ roots above the real axis and $N$ below when $\Xi$ and $\Xi'$ are linearly independent. The determinant is simply
\[
L(x,\Xi+\tau\Xi') = \abs{\Xi+\tau\Xi'}^{4N} = (\abs{\Xi}^2 + 2\tau\langle\Xi,\Xi'\rangle + \tau^2\abs{\Xi'}^2)^{2N}.
\]
By the linear independence of $\Xi$ and $\Xi'$ and the Cauchy-Schwarz inequality, the roots of this polynomial are split half above the real axis and half below. Hence the system is properly elliptic.

To show the system of boundary values satisfies the complementing condition, we argue as follows. At a point $iy$ in the boundary of our domain, the normal vector is $n = (-1,0)$ and any tangent vector $\zeta$ has the form $\zeta=(0,\zeta_2)$. The roots of $L(iy,\zeta+\tau n)$ with positive imaginary part are
\[
\tau_s^+ = i\zeta_2 \qquad s=1,\ldots,2N,
\]
assuming $\zeta_2>0$. Now
\[
L_0^+ = (\tau-i\zeta_2)^{2N}.
\]
The matrix adjoint of $L(iy,\zeta+\tau n) = \abs{\zeta+\tau n}^2Id_{2n}$ is $\abs{\zeta+\tau n}^{2N-1}Id_{2N}$. Hence the matrix $Q$ is
\[
Q(iy,\zeta;\tau) = \abs{\zeta+\tau n}^{2N-1}B
\]
The linear independence of the rows of $Q$ modulo $L_0^+$ is equivalent to the linear independence of the rows of $B$ modulo $\abs{\zeta+\tau n}$.

The determinant of $B$ can be computed as
\[
det(B) = N\tau.
\]
As this quantity is invertible modulo $\abs{\zeta+\tau n}$,, the rows of $Q$ are linearly independent modulo $L_0^+$. Hence the system of boundary conditions satisfies the complementing condition, and the regularity of \cite{M}, Theorem 6.8.2 applies.
\end{proof}

\end{document}